\documentclass[12pt]{article}

\usepackage{a4wide}
\usepackage{fancyhdr}
\usepackage{graphicx}
\usepackage{amsfonts,amsmath,amssymb,amsthm}
\usepackage{ifthen}

\newcommand{\Dchaintwo}[4]{
\rule[-3\unitlength]{0pt}{8\unitlength}
\begin{picture}(14,5)(0,3)
\put(1,2){\ifthenelse{\equal{#1}{l}}{\circle*{2}}{\circle{2}}}
\put(2,2){\line(1,0){10}}
\put(13,2){\ifthenelse{\equal{#1}{r}}{\circle*{2}}{\circle{2}}}
\put(1,5){\makebox[0pt]{\scriptsize #2}}
\put(7,4){\makebox[0pt]{\scriptsize #3}}
\put(13,5){\makebox[0pt]{\scriptsize #4}}
\end{picture}}

\newcommand\ord{{\operatorname{ord}}}

\newcommand{\BV}{\mathfrak{B}(V)}

\newcommand\Z{\mathbb{Z}}

\newcommand\Ly{\mathcal{L}}

\newcommand\SupW{[\Ly]^{(\N)}}

\newcommand\SupWL{[L]^{(\N)}}

\newcommand\N{\mathbb{N}}
\renewcommand{\k}{\Bbbk}
\newcommand\G{\Gamma}
\newcommand\Gh{\widehat{\Gamma}}

\renewcommand\o{\otimes}       
\newcommand{\la}{\langle}      \newcommand{\ra}{\rangle}
  
\newcommand{\red}[1]{c_{ {#1}}}     
\newcommand{\redbr}[2]{c_{{(#1|#2)}}}
\newcommand{\redh}[1]{d_{{#1}}} 
\newcommand\kX{\k\la X\ra}

\newcommand\X{\la X\ra}

\newcommand\Sh[3]{\mathrm{Sh}(#1)=(#2|#3)}
\newcommand\nSh[3]{\operatorname{Sh}(#1)\neq(#2|#3)}

\numberwithin{equation}{section}
\theoremstyle{plain}
\newtheorem{thm}{Theorem}[section]
\newtheorem{cor}[thm]{Corollary}
\newtheorem{lem}[thm]{Lemma}
\newtheorem{prop}[thm]{Proposition}

\theoremstyle{definition} 
\newtheorem{defn}[thm]{Definition}

\newtheorem{rem}[thm]{Remark}
\newtheorem{exs}[thm]{Examples}
\newtheorem{ex}[thm]{Example}

\title{On the presentation of pointed Hopf algebras\footnote{This work is part of the author's PhD thesis written under the supervision of Professor H.-J.~Schneider.}}

\author{Michael Helbig\footnote{eMail: \texttt{michael@helbig123.de}}}
\date{\today}

\begin{document}

\maketitle

\begin{abstract}
\noindent
We give a presentation in terms of generators and relations of Hopf algebras generated by skew-primitive elements and abelian group of group-like elements with action given via characters. This class of pointed Hopf algebras has shown great importance in the classification theory and can be seen as generalized quantum groups.
As a consequence we get an analog presentation of Nichols algebras of diagonal type.
\newline

\end{abstract}

\section*{Introduction}
Many famous examples of Hopf algebras generated by skew-primitive elements and abelian group of group-like elements with action given via characters are known: the universal enveloping algebras $U(\mathfrak g)$ of  a semi-simple complex Lie algebra $\mathfrak g$, their $q$-deformations  $U_q (\mathfrak g)$ the quantum groups of Drinfel'd and Jimbo, and the small quantum groups $u_q (\mathfrak g)$ of Lusztig \cite{Luszt1,Luszt2}, also called Frobenius-Lusztig kernels.
These are all pointed, i.e., all its simple subcoalgebras are one-dimensional, or equivalently, the coradical equals the group algebra of the group of group-like elements. 

Moreover, this class of  pointed Hopf algebras is very important in the classification theory: It is conjectured that any finite-dimensional pointed Hopf algebra over the complex numbers with abelian coradical is of that type. Recently Andruskiewitsch and Schneider  \cite{AS-Class} have proven this with some restriction on the order of the group of group-like elements. Further it is true in other special cases: for cocommutative Hopf algebras we have the Cartier-Kostant-Milnor-Moore theorem of around 1963, if the Hopf algebra has rank one see \cite{Krop2006214}, and if the dimension of the Hopf algebra is some power of a prime see  \cite{AS-p3, AS-CartMatr, AS-A2,Grana-p5}. However, the setting in the general situation is more complicated and new phenomena appear; for concrete examples see \cite{Helbig-PhD, Helbig-Lift}.

In  Theorem \ref{PropIdealCharHopfAlg} we give a structural description of Hopf algebras generated by skew-primitive elements and abelian group of group-like elements with action given via characters, in terms of generators and relations - we also get an analog statement for Nichols algerbas of diagonal type: At first we define a smash product of a free algebra and a group algebra, which is the prototype of these Hopf algebras. Then we show that any such Hopf algebra is a quotient of this prototype. The main point is the construction of the ideal, where
we use the theory of Lyndon words. We then use a result by Kharchenko \cite{KhPBW}. Also the generators of the ideal build up a Gr\"{o}bner basis, see also \cite{Helbig-PBW}. The knowledge of this presentation is important, if one wants to determine the liftings of Nichols algebras of diagonal type \cite{Helbig-Lift}, which is part of the lifting method by Andruskiewitsch and Schneider \cite{AS-p3}.

The paper is organized as follows: In Section \ref{SectqCommutCalc} we recall the general calculus for $q$-commutators in an arbitrary algebra of \cite{Helbig-PBW}.
Then in Section \ref{SectLyndW} we give a short account to the theory of Lyndon words, super letters and super words; super letters are iterated $q$-commutators and super words are products of super letters. We show that the set of all super words can be seen indeed as a set of words, i.e., as a free monoid. This is a consequence of Proposition \ref{LemSupWordsUnique}. Finally, we formulate in Section 3 our main result, and give some applications and classical examples in Section 4; more complicate examples are found in \cite{Helbig-PhD, Helbig-Lift}.

Throughout the paper let $\k$ denote a field  of arbitrary characteristic $\operatorname{char}\k=p\ge 0$, unless stated otherwise.
    \section{$q$-commutator calculus}\label{SectqCommutCalc}
In this section let $A$ denote an arbitrary algebra over $\k$.  For all $a,b\in A$ and $q\in \k$ we define the \emph{$q$-commutator}
$$
[a,b]_{q} := ab-qba.
$$
The $q$-commutator is bilinear. If $q=1$ we get the classical commutator of an 
algebra. If $A$ is graded and $a,b$ are homogeneous 
elements, then there is a natural choice for the $q$. 
We are interested in the following special case:
\begin{ex}\label{qCommutExkX}\label{qCommutEx} 
Let $\theta\ge 1$, $X=\{ x_1,\ldots,x_{\theta}\}$, $\X$ the free monoid and $A=\kX$ the free $\k$-algebra. For an abelian group $\G$ let $\widehat{\G}$ be the character group, $g_1,\ldots,g_{\theta}\in\G$ and $\chi_1,\ldots,\chi_{\theta}\in \widehat{\G}$. If we define the two monoid maps 
$$
\deg_{\G}:\X\rightarrow \G,\ \deg_{\G}(x_i):=g_i\quad\text{and}\quad \deg_{\Gh}:\X\rightarrow \Gh,\ \deg_{\Gh}(x_i):=\chi_i,
$$ 
for all $1\le i\le\theta$, then $\kX$ is $\G$- and $\Gh$-graded. 
Let $a\in \kX$ be $\G$-homogeneous and $b\in \kX$ be $\widehat{\G}$-homogeneous. We set 
$$
g_a:=\deg_{\G}(a), \quad \chi_b:=\deg_{\Gh}(b),\quad\text{and}\quad q_{a,b}:=\chi_b(g_a).
$$ 
Further we define $\k$-linearly on $\kX$ the $q$-commutator
\begin{align}\label{qCommutExkXDefn}
 [a,b] := [a,b]_{q_{a,b}}.
\end{align}
Note that $q_{a,b}$ is a bicharacter on the homogeneous elements
and depends only on the values 
$$
q_{ij}:=\chi_j(g_i)\text{ with }1\le i,j\le\theta.$$
For example $[x_1,x_2]=x_1x_2-\chi_2(g_1)x_2x_1=x_1x_2-q_{12}x_2x_1$. Further if $a,b$ are $\Z^{\theta}$-homogeneous they are both $\G$- and $\Gh$-homogeneous. In this case we can build iterated $q$-commutators, like $\bigl[x_1,[x_1,x_2]\bigr] =x_1[x_1,x_2]-\chi_1\chi_2(g_1)[x_1,x_2]x_1= x_1[x_1,x_2]-q_{11}q_{12}[x_1,x_2]x_1$.
\end{ex}

Later we will deal with algebras which still are  $\widehat{\G}$-graded,  but not $\G$-graded such that Eq.~\eqref{qCommutExkXDefn} is not well-defined. However, the $q$-commutator calculus, which we next want to develop, will be a major tool for our calculations such that we need the general definition with the $q$ as an index.

\begin{prop}\emph{\cite[Prop.~1.2]{Helbig-PBW}}\label{PropqCommut} For all $a,b,c,a_i,b_i\in A$, $q,q',q'',q_i\in \k$ with $1\le i\le n$ we have:

\emph{(1) $q$-derivation properties:}
\begin{align*}
   &[a,bc]_{qq'}=[a,b]_{q} c + q b [a,c]_{q'}, \qquad
   [ab,c]_{qq'}=a [b,c]_{q'} + q' [a,c]_{q} b,\\
   &[a,b_1\ldots b_n]_{q_1\ldots q_n}=\sum_{i=1}^{n}q_1\ldots q_{i-1} b_1\ldots b_{i-1}[a,b_i]_{q_i}b_{i+1}\ldots b_n,\\ 
   &[a_1\ldots a_n,b]_{q_1\ldots q_n}=\sum_{i=1}^{n}q_{i+1}\ldots q_{n} a_1\ldots a_{i-1}[a_i,b]_{q_i}a_{i+1}\ldots a_n.
\end{align*}

\emph{(2) $q$-Jacobi identity:}
\begin{align*}
   \bigl[[a,b]_{q'}, c\bigr]_{q''q}&=\bigl[a,[b,c]_{q}\bigr]_{q'q''} -q' b [a,c]_{q''}+ q[a,c]_{q''} b.
\end{align*}
\end{prop}
\begin{rem}
If we are in the situation of Example \ref{qCommutExkX} and assume that the elements are homogeneous, we can replace the arbitrary commutators by Eq.~\eqref{qCommutExkXDefn} and also
replace the general $q$'s above in the obvious way; e.g., in the first one of (1) set $q=q_{a,b}$, $q'=q_{a,c}$ and in (3), (4) $\zeta=q_{b,b}$ resp.~$\zeta=q_{a,a}$.
\end{rem}

\section{Lyndon words and $q$-commutators}\label{SectLyndW}

In this section we recall the theory of Lyndon words \cite{Lothaire,Reut} as far as we are concerned and then introduce the notion of super letters and super words \cite{KhPBW}. We want to emphasize that the set of all super words can be seen indeed as a set of words (more exactly as a free monoid, see \cite{Helbig-PBW}), which is a consequence of Proposition \ref{LemSupWordsUnique}. 
Moreover, we introduce a well-founded ordering  of the super words.

\subsection{Words and the lexicographical order}
Let $\theta\ge 1$, $X=\{x_1,x_2,\ldots,x_{\theta}\}$ be a finite totally ordered set by $x_1<x_2<\ldots <x_{\theta}$, and $\X$ the free monoid; we think of $X$ as an alphabet and of $\X$ as the words in that alphabet including the empty word $1$. For a word $u=x_{i_1}\ldots x_{i_n}\in\X$ we define $\ell(u):=n$ and call it the \emph{length} of $u$. 

The \emph{lexicographical order} $\le$ on $\X$ is defined for $u,v\in\X$ by $u<v$ if and only if 
either $v$ begins with $u$, i.e., $v=uv'$ for some $v'\in\X\backslash\{1\}$, or if there are $w,u',v'\in \X$, $x_i,x_j\in X$ 
such that $u=wx_iu'$, $v=wx_jv'$ and $i<j$. E.g., $x_1<x_1x_2<x_2$. 
This  order $<$ is stable by left, but in general not stable by right multiplication: $x_1<x_1x_2$ but $x_1x_3>x_1x_2x_3$. Still we have:
\begin{lem}\label{lemord}
Let $v,w\in\X$ with $v<w$. Then:
\begin{enumerate}
\item[\emph{(1)}] $uv<uw$ for all $u\in\X$.
\item[\emph{(2)}] If $w$ does not begin with $v$, then $vu<wu'$ for all $u,u'\in\X$.
\end{enumerate}
\end{lem}

\subsection{Lyndon words and the Shirshov decomposition}
A word $u\in\X$ is called a \emph{Lyndon word} if $u\neq 1$ and $u$ is smaller than any of its proper endings, i.e., for all $v,w\in\X\backslash\{1\}$ such that $u=vw$ we have $u<w$. We denote by 
$$
\Ly:=\{u\in\X\,|\, u \text{ is a Lyndon word}\}
$$ the set of all Lyndon words. For example $X\subset\Ly$, but $x_i^n\notin \Ly$ for all $1\le i\le \theta$ and $n\ge 2$. 
Moreover, if $i<j$ then
$x_i^n x_j^m\in\Ly$ for $n,m\ge 1$, e.g. $x_1x_2$, $x_1x_1x_2$, $x_1x_2x_2$, 
 $x_1x_1x_2x_2$; also 
$x_i(x_ix_j)^n\in\Ly$ for any $n\in\N$, e.g. $x_1x_1x_2$, $x_1x_1x_2x_1x_2.$

For any $u\in\X\backslash X$ we call the decomposition $u=vw$ with $v,w\in \X\backslash\{1\}$ such that $w$ is the minimal (with respect to the lexicographical order) ending  the \emph{Shirshov decomposition} of the word $u$. We will write in this case $$\Sh{u}{v}{w}.$$ E.g., $\Sh{x_1x_2}{x_1}{x_2}$, $\Sh{x_1x_1x_2x_1x_2}{x_1x_1x_2}{x_1x_2}$, $\nSh{x_1x_1x_2}{x_1x_1}{x_2}$.

If $u\in\Ly\backslash X$, this is equivalent to $w$ is the longest proper ending of $u$ such that $w\in\Ly$. Moreover we have another characterization of the Shirshov decomposition of Lyndon words:

\begin{thm}\label{ShirChar}\emph{\cite[Prop. 5.1.3, 5.1.4]{Lothaire}}
Let $u\in\X\backslash X$ and $u=vw$ with $v,w\in \X$. Then the following are equivalent:
\begin{enumerate}
 \item[\emph{(1)}]  $u\in\Ly$ and $\Sh{u}{v}{w}$.
 \item[\emph{(2)}]  $v,w\in\Ly$ with $v<u<w$ and either $v\in X$ or else if $\Sh{v}{v_1}{v_2}$ then $v_2\ge w$.
\end{enumerate}
\end{thm}

With this property we see that any Lyndon word is a product of two other Lyndon words of smaller length. Hence we get every Lyndon word by starting with $X$ and concatenating inductively each pair of Lyndon words $v,w$ with $v<w$.
\begin{defn}
We call a subset $L\subset \Ly$  \emph{Shirshov closed}
if  $X\subset L$, and  for all $u\in L$ with $\Sh{u}{v}{w}$ also $v,w\in L$.
\end{defn}

For example $\Ly$ is Shirshov closed, and if $X=\{x_1,x_2\}$, then $\{x_1,x_1x_1x_2,x_2\}$ is  not Shirshov closed,  whereas $\{x_1,x_1x_2,x_1x_1x_2,x_2\}$ is. Later we will need the following: 
\begin{lem}\emph{\cite[Lem.~4]{KhPBW}}\label{LemZerleg}
Let $u,v\in\Ly$ and $u_1,u_2\in\X\backslash\{1\}$ such that $u=u_1u_2$ and $u_2<v$. Then we have
$$
uv<u_1v<v\mbox{ and }uv<u_2v<v.
$$
\end{lem}
\subsection{Super letters and super words} Let the free algebra $\kX$ be graded as in Example \ref{qCommutExkX}. For any $u\in\Ly$ we define recursively on $\ell(u)$ the map 
\begin{align}\label{DefnSuperLett}
[\,.\,]:\Ly\rightarrow\kX,\quad u \mapsto [u].
\end{align}

\noindent
If $\ell(u)=1$, then set $[x_i]:=x_i$ for all $1\le i\le\theta$. Else if $\ell(u)>1$ and $\Sh{u}{v}{w}$ we define $[u]:=\bigl[[v],[w]\bigr]$. This map is well-defined since inductively all $[u]$ are $\Z^{\theta}$-homogeneous such that we can build iterated $q$-commutators; see Example \ref{qCommutExkX}. The elements $[u]\in\kX$ with $u\in\Ly$ are called \emph{super letters}. E.g. $[x_1x_1x_2x_1x_2]=\bigl[[x_1x_1x_2],[x_1x_2]\bigr]=\bigl[[x_1,[x_1,x_2]],[x_1,x_2]\bigr]$.
If $L\subset \Ly$ is Shirshov closed then the subset of $\kX$
$$
[L]:=\bigl\{[u]\,\big|\, u\in L\bigr\}
$$ 
is a set of iterated $q$-commutators. Further 
$
[\Ly]=\bigl\{[u]\,\big|\, u\in\Ly\bigr\}
$
is the set of all super letters and the map 
$[\,.\,]:\Ly\rightarrow[\Ly]$ is a bijection, which follows from Lemma \ref{LemLexicoSupWords} below. Hence we can define an order $\le$ of the super letters $[\Ly]$ by 
$$
[u]< [v]:\Leftrightarrow u<v,
$$ thus $[\Ly]$ is a new alphabet containing the original alphabet $X$; so the name ``letter'' makes sense. Consequently, products of super letters are called \emph{super words}. We denote 
$$
[\Ly]^{(\N)}:=\bigl\{[u_1]\ldots [u_n]\,\bigl|\,n\in\N,\, u_i\in\Ly\bigr\}
$$
the subset of $\kX$ of all super words. In order to define a lexicographical order on $\SupW$, we need to show that an arbitrary super word has a unique factorization in super letters. This is not shown in \cite{KhPBW}.

For any word $u=x_{i_1}x_{i_2}\ldots x_{i_n}\in\X$ we define the \emph{reversed word} 
$
\overleftarrow{u}:=x_{i_n}\ldots x_{i_2}x_{i_1}.
$
Clearly, $\overleftarrow{\overleftarrow{u}}=u$ and  $\overleftarrow{uv}=\overleftarrow{v}\overleftarrow{u}$. Further for any $a=\sum \alpha_i u_i\in\kX$ we call the lexicographically smallest word of the $u_i$ with $\alpha_i\neq 0$ the \emph{leading word} of $a$ and further define $\overleftarrow{a}:=\sum \alpha_i \overleftarrow{u_i}$.

\begin{lem}\label{LemLexicoSupWords} Let $u\in\Ly\backslash X$. Then there exist $n\in\N$, $u_i\in\X$, $\alpha_i\in\k$ for all $1\le i\le n$ and $q\in\k^\times$ such that 
$$
[u]=u+\sum_{i=0}^n\alpha_iu_i+ q\overleftarrow{u}\quad \text{and}\quad \overleftarrow{[u]}=\overleftarrow{u}+\sum_{i=0}^n\alpha_i\overleftarrow{u_i}+ q u.
$$
Moreover, $u$ is the leading word of both $[u]$ and $\overleftarrow{[u]}$.
\end{lem}
\begin{proof} We proceed by induction on $\ell(u)$. If $\ell(u)=2$, then $u=x_ix_j$ for some $1\le i<j\le \theta$ and $[u]=[x_ix_j]=x_ix_j- q_{ij}x_jx_i=u-q_{ij}\overleftarrow{u}$. Let $\ell(u)>2$, $\Sh{u}{v}{w}$ and $[u]=[v][w]-q_{vw}[w][v]$. By induction
\begin{align*}
 [v]=v+\sum_{i}\beta_iv_i+ q\overleftarrow{v}&\quad\text{ and }\quad \overleftarrow{[v]}=\overleftarrow{v}+\sum_{i}\beta_i\overleftarrow{v_i}+ q v,\quad\text{ resp.}\\
 [w]=w+\sum_{j}\gamma_iw_i+ q'\overleftarrow{w}&\quad\text{ and }\quad \overleftarrow{[w]}=\overleftarrow{w}+\sum_{i}\gamma_i\overleftarrow{w_i}+ q' w
\end{align*}
with $q,q'\neq 0$ and leading word $v$ resp.~$w$. Hence $[v][w]$ and $\overleftarrow{[v]}\overleftarrow{[w]}$ resp.~$[w][v]$ and $\overleftarrow{[w]}\overleftarrow{[v]}$ have the leading words $vw$ resp.~$wv$. Since $u$ is Lyndon we get $u=vw<wv$, thus the leading word of $[u]$ and $\overleftarrow{[u]}$ is $u$ and further they are of the claimed form.
\end{proof}

\begin{prop}\label{LemSupWordsUnique} Let $u_1,\ldots,u_n,v_1,\ldots,v_m\in\Ly$. If $[u_1][u_2]\ldots[u_n] = [v_1][v_2]\ldots[v_m]$,  then $m=n$ and $u_i=v_i$ for all $1\le i\le n$.
\end{prop}

\begin{proof}
Induction on $\operatorname{max}\{m,n\}$, we may suppose $m\le n$. If $n=1$ then also $m=1$, hence $[u_1]=[v_1]$ and both have the same leading word $u_1=v_1$.

Let $n>1$: By Lemma \ref{LemLexicoSupWords} $[u_1]\ldots[u_n]=[v_1]\ldots[v_m]$ has the leading word $u_1\ldots u_n=v_1\ldots v_m$ and 
$$
\overleftarrow{[u_n]}\ldots \overleftarrow{[u_1]}=\overleftarrow{[u_1]\ldots[u_n]}=\overleftarrow{[v_1]\ldots[v_m]}=\overleftarrow{[v_m]}\ldots\overleftarrow{[v_1]}
$$
has the leading word $u_n\ldots u_1=v_m\ldots v_1$. 

If $\ell(u_1)\ge\ell(v_1)$, then $u_1=v_1u$  and $u_1=u'v_1$ for some $u,u'\in\X$. If $u,u'\neq 1$, we get the contradiction $v_1<v_1u=u'v_1<v_1$, since $u_1$ is Lyndon. Else if $\ell(u_1)<\ell(v_1)$, it is the same argument using that $v_1$ is Lyndon. Hence $u_1=v_1$ and by induction the claim follows.
\end{proof}

Now the lexicographical order on all super words $[\Ly]^{(\N)}$, as defined above on regular words, is well-defined. We denote it also by $\le$.

\subsection{A well-founded ordering of super words}\label{SectWellFoundOrder} The \emph{length} of a super word $U=[u_1][u_2]\ldots[u_n]\in\SupWL$ is defined as 
$
\ell(U):=\ell(u_1u_2\ldots u_n).
$
\begin{defn}
For $U, V\in \SupW$ we define $U\prec V$ by 
\begin{itemize}
\item $\ell(U)<\ell(V)$, or
\item $\ell(U)=\ell(V)$  and $U>V$ lexicographically in $\SupW$.
\end{itemize}
\end{defn}
This defines a total ordering of $\SupW$ with minimal element $1$. As $X$ is assumed to be finite, there are only finitely many super letters of a given length. Hence every nonempty subset of $\SupW$ has a minimal element, or equivalently, $\preceq$ fulfills the descending chain condition: $\preceq$ is \emph{well-founded}, making way for inductive proofs on $\preceq$.

  \section{A class of pointed Hopf algebras}\label{SectCharHA}
In this chapter we deal with a special class of pointed Hopf algebras. Let us recall the notions and results of \cite[Sect.~3]{KhPBW}: Let $\theta\ge 1.$ A Hopf algebra $A$ is called a \emph{character Hopf algebra} if it is generated as an algebra by
elements $a_1,\ldots ,a_{\theta}$ and an abelian group $G(A)=\G$ of all group-like elements such that  for all $1\le i\le \theta$ there are $g_i\in\G$ and $\chi_i\in\Gh$ with
\begin{align*}
\Delta(a_i)=a_i\o 1 +g_i\o a_i\qquad \text{and}\qquad
 ga_i = \chi_i(g)a_i g.
\end{align*}
As mentioned in the introduction this covers a wide class of examples of Hopf algebras. The minimal number $\theta$ such that $a_1,\ldots ,a_{\theta}$ (also with renumbering) and $\G$ generate $A$ is called the \emph{rank} of $A$.

The aim of this section is to construct for any \emph{character Hopf algebra} $A$ a smash product $\kX\#\k[\G]$ together with an ideal $I$ such that
$
A\cong (\kX \# \k[\G])/I.
$
Note that any character Hopf algebra is $\Gh$-graded by 
$
A=\oplus_{\chi\in\Gh}A^\chi\quad\text{with} \quad A^\chi:=\{a\in A\ |\ ga=\chi(g)ag\},
$ 
since $A$ is genereated by $\Gh$-homogeneous elements, and elements of different $A^\chi$ are linearly independent. 

\subsection{PBW basis in hard super letters}\label{SubSectHardSupLettKh} 
Let from now on  $A$ be  a character Hopf algebra. The algebra map $$\kX\rightarrow A,\quad x_i\mapsto a_i$$ allows to identify elements of $\kX$ with elements of $A$: By abuse of language we will write for the image of $a\in \kX$ also $a$. Further let $\kX$ be $\G$-, $\Gh$-graded and $q_{u,v}$ as in Example \ref{qCommutExkX} with the $g_i$ and $\chi_i$ above. 
Then a super letter $[u]\in A$ is called \emph{hard} if it is not a linear combination of 
\begin{itemize}
 \item $\quad U=[u_1]\ldots[u_n]\in\SupW$ with $n\ge 1$, $\ell(U)=\ell(u)$, $u_i> u$ for all $1\le i\le n$, and
\item $\quad Vg$ with $V\in\SupW$, $\ell(V)<\ell(u)$ and $g\in\G$.
\end{itemize}
Note that if $[u]$ is hard and $\Sh{u}{v}{w}$,  then also $[v]$ and $[w]$ are hard; this follows from \cite[Cor.~2]{KhPBW}. We may assume that $a_1,\ldots,a_\theta$ are hard, otherwise $A$ would be generated by $\G$ and a proper subset of $a_1,\ldots ,a_{\theta}$. But this says that the set of all hard super letters is Shirshov closed.

For any hard $[u]$ we define $N_u'\in\{2,3,\ldots,\infty\}$ as the minimal $r\in\N$ such that $[u]^r$ is  not a linear combination of 
\begin{itemize}
 \item $\quad U=[u_1]\ldots[u_n]\in\SupW$ with $n\ge 1$, $\ell(U)=r\ell(u)$, $u_i> u$ for all $1\le i\le n$, and
\item $\quad Vg$ with $V\in\SupW$, $\ell(V)<r\ell(u)$ and $g\in\G$.
\end{itemize}

\begin{thm}\emph{\cite[Thm.~2, Lem.~13]{KhPBW}}\label{ThmKh} Let $A$ be a character Hopf algebra. Then the set of all $$[u_1]^{r_1}[u_2]^{r_2}\ldots [u_t]^{r_t}g$$ with $t\in\mathbb N$, $[u_i]$ is hard, $u_1>\ldots>u_t$, $0< r_i<N_{u_i}'$, $g\in\G$, forms a $\k$-basis of $A$. 

Further, for every hard super letter $[u]$ with $N_u'<\infty$ we have $\ord q_{u,u}=N_u'$ if $\operatorname{char}\k=0$ resp.~$p^k\ord q_{u,u}=N_u'$ for some $k\ge 0$ if $\operatorname{char}\k=p>0$.
\end{thm}
We now generally construct a smash product $\kX\# \k[\G]$ with an ideal $I$. 

\subsection{Prototype: The smash product $\kX\# \k[\G]$}
 Let $\kX$ be $\G$- and $\Gh$-graded as in Example \ref{qCommutExkX}, and $\k[\G]$ be endowed with the usual bialgebra structure $\Delta(g)=g\o g$ and $\varepsilon(g)=1$ for all $g\in\G$.
Then we define 
$$
g\cdot x_i := \chi_i(g)x_i,\ \text{ for all }1\le i\le\theta.
$$
In this case, $\kX$ is a $\k[\G]$-module algebra 
and we calculate $gx_i=\chi_i(g)x_ig$, $gh=hg=\varepsilon(g)hg$ in $\kX\# \k[\G]$.
Thus  $x_i\in(\kX\# \k[\G])^{\chi_i}$ and $\k[\G]\subset (\kX\# \k[\G])^\varepsilon$ and in this way 
$
\kX\# \k[\G]=\oplus_{\chi\in\Gh}(\kX\# \k[\G])^\chi.
$
This $\Gh$-grading extends the $\Gh$-grading of $\kX$ in Example \ref{qCommutExkX} to $\kX\# \k[\G]$.

Further   $\kX\# \k[\G]$ is a Hopf algebra with structure determined by 
\begin{align*}
 \Delta(x_i):=x_i\o 1 + g_i\o x_i\qquad\text{and}\qquad  \Delta(g):=g\o g,
\end{align*}
for all $1\le i \le\theta$ and $g\in\G$. Thus our prototype is indeed a character Hopf algebra. Other character Hopf algebras arise from certain quotients of this prototype, as an example consider the following:

\subsection{Motivation: Quantum groups $U_q(\mathfrak{sl}_2)$ and $u_q(\mathfrak{sl}_2)$}\label{SectionMotivation}
Let $q\in\k^{\times}\backslash\{\pm 1\}$, then we define like in \cite[Sect.~VI,VII]{Kassel}
\begin{align*}
U_q(\mathfrak{sl}_2)
         &:=\k\Bigl\langle E,F,K,K^{-1}\ \big|\  KK^{-1}=K^{-1}K=1,\ KEK^{-1}=q^2E,\ KFK^{-1}=q^{-2}F,\\
&\phantom{=\k\Bigl\langle E,F,K,K^{-1}\ \big|}\ \ EF-FE=\frac{K-K^{-1}}{q-q^{-1}}\Bigr\rangle,\\
\Delta(E)&=E\o K + 1\o E,\quad \Delta(F)=F\o 1 + K^{-1}\o F,\quad \Delta(K^{\pm 1})=K^{\pm 1}\o K^{\pm 1},
\end{align*}
which is a character Hopf algebra, too. Let us rewrite this presentation to our conventions:
\begin{ex}\label{exampleQuantumGroupSL2}
\emph{Quantum group}.
Set 
$\G:=\langle g,g^{-1} \ |\ gg^{-1}=g^{-1}g= 1\rangle\cong \Z$, $g_1:=g_2:=g$, and $\chi_1(g):=q^{-2}$, $\chi_2(g):=q^2$.
Then
$$
U_q(\mathfrak{sl}_2) \cong \bigl(\k\la x_1,x_2\ra\#\k[\G]\bigr)/\bigl([x_1x_2]-(1-g^2)\bigr).
$$
\end{ex}
\begin{proof}The isomorphism of Hopf algebras is the map from $U_q(\mathfrak{sl}_2)$ to the right-hand side
which sends $E\mapsto \frac{q^2}{q-q^{-1}}x_1g^{-1}$, $F\mapsto x_2$, $K\mapsto g^{-1}$, $K^{-1}\mapsto g$.
\end{proof}

A finite-dimensional version is the following example:
\begin{ex}\label{exampleFrobLusztKernelSL2}
\emph{Frobenius-Lusztig kernel}.
Let $q\in\k^{\times}$ with odd $\ord q=N>2$, $\G:=\langle g \ |\ g^N=1\rangle\cong \Z/(N)$, $g_1:=g_2:=g$, and $\chi_1(g):=q^{-2}$, $\chi_2(g):=q^2$. Then
\begin{align*}
 u_q(\mathfrak{sl}_2)\cong  \bigl(\k\la x_1,x_2\ra\#\k[\G]\bigr)/\bigl(\ [x_1x_2]-(1-g^2),\ x_1^{N},\ x_2^{N}\ \bigr).
\end{align*}
\end{ex}

Note that in the above examples there are relations involving super letters and powers of super letters. Next we construct the ideals in the general setting:

\subsection{Ideals associated to Shirshov closed sets}\label{SectIdealOfCharHA} 
In this subsection we fix a Shirshov closed $L\subset\Ly$. 
We want to introduce the following notation 
for an $a\in \kX\# \k[\G]$ and  
$W\in\SupW$: 
We will write $a\prec_L W$ (resp. $a\preceq_L W$), if 
$a$ is a linear combination of 
\begin{enumerate}
 \item[\textbullet] $U\in\SupWL$ with $\ell(U)=\ell(W)$, $U>W$ (resp. $U\ge W$), and
\item[\textbullet]  $Vg$ with $V\in\SupWL$, $g\in\G$,  $\ell(V)<\ell(W)$.
\end{enumerate}

Furthermore,  we set for each $u\in L$
either $N_u:=\infty$ or $N_u:=\ord q_{u,u}$  (resp.~$N_u:=p^k\ord q_{u,u}$ with $k\ge 0$ if $\operatorname{char} \k=p>0$) and we want to distinguish the following two sets of words depending on $L$:
\begin{align*}
C(L) &:= \bigl\{w\in\X\backslash L \ |\ \exists u,v\in L:  w=uv,\ u<v,\  \text{and}\ \Sh{w}{u}{v}\bigr\},\\
D(L) &:= \bigl\{u\in L \ |\ N_u<\infty\}.
\end{align*}
Note that $C(L)\subset \Ly$ and $D(L)\subset L\subset \Ly$ are sets of Lyndon words.

Moreover, let $\red{w} \in (\kX\# \k[\G])^{\chi_{w}}$ for all $w\in C(L)$ such that $\red{w}\prec_L [w]$;
and let $\redh{u} \in (\kX\# \k[\G])^{\chi_u^{N_u}}$ for all $u\in D(L)$ such that $\redh{u}\prec_L [u]^{N_u}$. 
Now we define the $\Gh$-homogeneous ideal:
\begin{defn}\label{DefnIdealLcd}
In the above setting let $I_{L,N,c,d}$ be the ideal of $\kX\# \k[\G]$ generated by the following elements:
\begin{align}
[w]         - \red{w}&  \qquad \text{for all } w\in C(L) ,\label{ThmPBWCritIdluv}\\
 [u]^{N_u}   - \redh{u}&  \qquad \text{for all }  u\in D(L).\label{ThmPBWCritIdluN}
\end{align}
For shortness we will write just $I$ for $I_{L,N,c,d}$.
\end{defn}

\begin{ex}Let $X=\{x_1,x_2\}$. If $L=X$ then $C(L)=\{x_1x_2\}$ and we have 
 Examples \ref{exampleQuantumGroupSL2} and \ref{exampleFrobLusztKernelSL2}. It is $N_1=N_2=\infty$ in Example \ref{exampleQuantumGroupSL2}  and $N_1=N_2=N$ in \ref{exampleFrobLusztKernelSL2}. Other examples  are treated in Sections \ref{SectExamplesRankOne} and \ref{SectExamplesRankTwo}. Furthermore we want to mention the list of  more complicated examples of \cite{Helbig-PhD, Helbig-Lift}.
\end{ex}

In the next Lemma we want to define $\redbr{u}{v}\in\kX\# \k[\G]$ for all $u,v\in L$ with $u<v$, such that $\bigl[[u],[v]\bigr]=\redbr{u}{v}$ modulo $I$. This shows that the relations $\bigl[[u],[v]\bigr]=\redbr{u}{v}$ with with $\nSh{uv}{u}{v}$ or $uv\in L$ are redundant modulo $I$. 
\begin{lem}\label{LemDefnfuv} Let $I'\subset \kX\#\k[\G]$ be  the ideal generated by the elements Eq.~\eqref{ThmPBWCritIdluv}. 
Then there are $\redbr{u}{v}\in(\kX\# \k[\G])^{\chi_{uv}}$ for all $u,v\in L$ with $u<v$ such that
\begin{enumerate}
\item[\emph{(1)}]  $\bigl[[u],[v]\bigr]-\redbr{u}{v}\in I'$,
\item[\emph{(2)}]  $\redbr{u}{v}\preceq_L [uv]$.
\end{enumerate}
The residue classes of $[u_1]^{r_1}[u_2]^{r_2}\ldots [u_t]^{r_t}g$ with $t\in\mathbb N$, $u_i\in L$, $u_1>\ldots>u_t$, $0< r_i<N_{u_i}$, $g\in\G$, $\k$-generate  $(\kX\# \k[\G])/I$.
\end{lem}

\begin{proof} 
For all $u,v\in L$ with $u<v$ and $\Sh{uv}{u}{v}$ we set
$$
\redbr{u}{v}:=\begin{cases}   [uv] ,   &\mbox{if }uv\in L, \\ 
                              \red{uv} ,    &\mbox{if }uv\notin L.
              \end{cases}
$$
We then proceed by induction on $\ell(u)$: If $u\in X$ then $\Sh{uv}{u}{v}$ by Theorem \ref{ShirChar} and by definition the claim is fulfilled. So let $\ell(u)>1$. Again if $\Sh{uv}{u}{v}$ then we argue as in the induction basis. Conversely, let $\nSh{uv}{u}{v}$,
and further $\Sh{u}{u_1}{u_2}$; then $u_2<v$ by Theorem \ref{ShirChar} and by Lemma \ref{LemZerleg} 
\begin{align}
u_1<u_1u_2=u<uv<u_2v  ,\mbox{ and }\;   uv<u_1v.
\label{Prooffuv1}
\end{align}
By induction hypothesis there is a 
$
\redbr{u_2}{v}=  \sum \alpha U +\sum\beta Vg 
$ 
(we omit the indices to avoid double indices) of $\Gh$-degree $\chi_{u_2v}$ with $U=[l_1]\ldots[l_n]\in\SupWL$, $\ell(U)=\ell(u_2v)$, $l_1\ge u_2v$, 
$V\in\SupWL$, $\ell(V)<\ell(u_2v)$, $g\in\G$ and $\bigl[[u_2],[v]\bigr]-\redbr{u_2}{v}\in I'$.  
Then 
$$
\bigl[[u_1],\redbr{u_2}{v}\bigr]
=
    \sum \alpha \bigl[[u_1],U\bigr]+\sum\beta \bigl[[u_1],Vg\bigr].
$$
Since $U$ is $\chi_{u_2v}$-homogeneous we can use the $q$-derivation property of Proposition \ref{PropqCommut} for the term
$$
\bigl[[u_1],U\bigr]
    = \sum_{i=1}^{n}q_{u_1,l_1\ldots l_{i-1}}[l_1]\ldots [l_{i-1}]\bigl[[u_1],[l_i]\bigr]
[l_{i+1}]\ldots [l_n].
$$
By assumption $u_2v\le l_1$, hence we deduce $uv<l_1$ and $u_1<l_1$ from Eq.~\eqref{Prooffuv1}; because of the latter inequality, by the induction hypothesis there is a $\chi_{u_1l_1}$-homogeneous $\redbr{u_1}{l_1}=\sum \alpha' U' +\sum\beta' V'g'$ with $U'\in\SupWL$, $\ell(U')=\ell(u_1l_1)$, $U'\ge [u_1l_1]$, $V'\in\SupWL$, $\ell(V')<\ell(u_1l_1)$, $g'\in\G$ and $\bigl[[u_1],[l_1]\bigr]-\redbr{u_1}{l_1}\in I'$. Since $u_2v\le l_1$ we have $[uv]=[u_1u_2v]\le [u_1l_1]\le U'$. We now define 
$\partial_{u_1}(\redbr{u_2}{v})$ 
$\k$-linearly by
\begin{align*}
&\partial_{u_1}(U):=\redbr{u_1}{l_1}[l_{2}]\ldots [l_n] + \sum_{i=2}^{n} q_{u_1,l_1\ldots l_{i-1}} [l_1]\ldots [l_{i-1}] \bigl[[u_1],[l_i]\bigr] [l_{i+1}]\ldots [l_n],\\
&\partial_{u_1}(Vg):= \bigl[[u_1],V\bigr]_{q_{u_1,u_2v}\chi_{u_1}(g)} g.
\end{align*}
Then
$\partial_{u_1}(\redbr{u_2}{v})\preceq_L [uv]$ with $\Gh$-degree $\chi_{uv}$. Moreover 
$\bigl[[u_1],\bigl[[u_2],[v]\bigr]\bigr]-\partial_{u_1}(\redbr{u_2}{v})\in I'$
, since $\bigl[[u_1],U\bigr]-\partial_{u_1}(U)\in I'$ and $\partial_{u_1}(Vg)=\bigl[[u_1],Vg\bigr]_{q_{u_1,u_2v}}$. 

Finally, because of $u_1<u<v$ there is again by induction assumption 
a $\redbr{u_1}{v}\preceq_L [u_1v]$, which is $\chi_{u_1v}$-homogeneous and $\redbr{u_1}{v}-\bigl[[u_1][v]\bigr]\in I'$ (moreover,  $u_1v>uv$ by Eq.~\eqref{Prooffuv1}). 
 We then define for $\nSh{uv}{u}{v}$
\begin{align}\label{RedCommutDefNotSh}
\redbr{u}{v} :=
   \partial_{u_1}(\redbr{u_2}{v}) +q_{u_2,v}\redbr{u_1}{v}[u_2] -q_{u_1,u_2}[u_2]\redbr{u_1}{v}.
\end{align}
We have $u_2>u$ since $u$ is Lyndon and $u$ cannot begin with $u_2$, hence $u_2>uv$ by Lemma \ref{lemord}. Thus 
$\redbr{u}{v}\prec_L [uv]$. 
Also $\deg_{\Gh}(\redbr{u}{v})=\chi_{uv}$ and by the  $q$-Jacobi identity of Proposition \ref{PropqCommut} we have
$\bigl[[u],[v]\bigr] -\redbr{u}{v}\in I'$.

For the last assertion it suffices to show that the residue classes of $[u_1]^{r_1}[u_2]^{r_2}\ldots [u_t]^{r_t}g$  $\k$-generate the residue classes of $\kX$ in $(\kX\# \k[\G])/I'$: this can be done as in the proof of \cite[Lem.~10]{KhPBW} by induction on $\preceq$ using (1),(2).
\end{proof}


\subsection{Structure of pointed Hopf algebras and Nichols algebras}

\begin{thm}\label{PropIdealCharHopfAlg}
If $A$ is a character Hopf algebra, then there is 
an ideal $I\subset \kX \# \k[\G]$ as in Definition \ref{DefnIdealLcd} 
such that
$$
A\cong (\kX \# \k[\G])/I.
$$
\end{thm}
\begin{proof}
Let $[L]$ be the set of hard super letters in $A$; then $L\subset\Ly$ is Shirshov closed as mentioned above. By Theorem \ref{ThmKh} the elements $[u_1]^{r_1}[u_2]^{r_2}\ldots [u_t]^{r_t}g$ with $t\in\mathbb N$, $u_i\in L$, $u_1>\ldots>u_t$, $0< r_i<N_{u_i}'$, $g\in\G$, form a $\k$-basis.  We consider the 
 $\k$-linear map 
$$
\phi: A\rightarrow \kX \# \k[\G], \quad [u_1]^{r_1}\ldots[u_t]^{r_t}g\mapsto [u_1]^{r_1}\ldots[u_t]^{r_t}g,
$$
and define $\red{w}:=\phi\bigl([w]\bigr)$ for all $w\in C(L)$, $\redh{u}:=\phi\bigl([u]^{N_u} \bigr)$ for all $u\in D(L)$, where $N_u:=N_u'$. Note that these elements are as stated in Lemma \ref{LemDefnfuv} since $[w]$ is not hard. 
Then we build the ideal $I\subset \kX \# \k[\G]$ like in Definition \ref{DefnIdealLcd} and there is the surjective Hopf algebra map 
$$
(\kX\# \k[\G])/I \rightarrow A,\quad x_i\mapsto a_i,\ g\mapsto g.
$$ 
By Lemma \ref{LemDefnfuv} the residue classes of $[u_1]^{r_1}\ldots[u_t]^{r_t}g$ $\k$-generate $(\kX\# \k[\G])/I$; they are linearly independent because  so are their images. Hence the map is an isomorphism. 
\end{proof}
One immediately gets the following result for Nichols algebras of diagonal type (for the definition of Nichols algebras we refer to \cite{AS-Pointed} or \cite{Helbig-Lift}):
\begin{cor}
Let $\BV$ be a Nichols algebra of diagonal type of a vector space $V$ with basis $X$. Then there is 
a homogeneous ideal $I\subset \kX$ as in Definition \ref{DefnIdealLcd} such that
$$
\BV \cong \kX /I.
$$
\end{cor}

\section{Application and examples}

We want to investigate the situation in more detail for character Hopf algebras of rank one and two for some fixed Shirshov closed subsets $L\subset\Ly$. Especially we want to treat liftings of Nichols algebras. Therefore we define the following scalars which will guarantee a $\Gh$-graduation:
\begin{defn}\label{DefnLiftCoeffic} Let $L\subset\Ly$. 
Then we define coefficients
 $\mu_u\in\k$ for all $u\in D(L)$, and $\lambda_{w}\in \k$ for all $w\in C(L)$ by 
\begin{align*}
\mu_u=0,\text{ if } g_u^{N_u}=1\text{ or }\chi_u^{N_u}\neq \varepsilon, \qquad\quad
\lambda_{w}=0,\text{ if } g_{w}=1 \text{ or }\chi_{w}\neq \varepsilon,
\end{align*}
and otherwise they can be chosen arbitrarily.
\end{defn}

\subsection{Pointed Hopf algebras of rank one} \label{SectExamplesRankOne}
\begin{prop}
Let $\operatorname{char} \k=p\ge 0$ and $A$ be a character Hopf algebra of rank one. Then either $A\cong \k[x_1]\# \k[\G]$, or
$$
A\cong (\k[x_1]\# \k[\G]) /\bigl( x_1^N-\redh{1} \bigr)
$$
for  $N=\ord q_{11}<\infty$ if $\operatorname{char} \k=0$ resp.~$N=p^k\ord q_{11}<\infty$ with $k\ge 0$ if $\operatorname{char} \k=p>0$ and $\redh{1} \in (\kX\# \k[\G])^{\chi_1^{N}}$ with $\redh{1} \prec_{\{x_1\}} x_1^{N}$.
\end{prop}
\begin{proof}
We have $\Ly=X=\{x_1\}$. Hence any Shirshov closed $L$ is equal to $\{x_1\}$ and thus $C(L)$ is the empty set. By Theorem \ref{PropIdealCharHopfAlg} we get the claim.
\end{proof}

\begin{cor}\label{CorNichAlgA1}
Let $\operatorname{char} \k=p\ge 0$ and $\BV$ be a Nichols algebra of diagonal type with $V=\k x_1$ a one-dimensional vector space. Then either $\BV\cong \k[x_1]$, or
$$
\BV\cong \k[x_1] /\bigl( x_1^N \bigr)
$$
for  $N=\ord q_{11}<\infty$ if $\operatorname{char} \k=0$ resp.~$N=p^k\ord q_{11}<\infty$ with $k\ge 0$ if $\operatorname{char} \k=p>0$.
\end{cor}
\begin{proof}
By Theorem \ref{PropIdealCharHopfAlg} the ideal is of the form $\bigl( x_1^N -\redh{1} \bigr)$ with $\redh{1} \in \kX^{\chi_1^{N}}$ and $\redh{1} \prec_{\{x_1\}} x_1^{N}$. Because $x_1^N$ is a primitive element we get $\redh{1}=0$, by definition of a Nichols algebra.
\end{proof}

If $\operatorname{char} \k=p$ not every finite-dimensional pointed Hopf algebra of rank one is a character Hopf algebra, since the group action is not necessarily via characters  \cite{Scherotzke20082889}. But we have: 
 \begin{prop}\emph{\cite[Thm.~1]{Krop2006214}}
If $\operatorname{char} \k=0$, then any finite-dimensional pointed Hopf algebra of rank one is a character Hopf algebra; moreover it is isomorphic to
$$
(\k[x_1]\# \k[\G]) /\bigl( x_1^N-\mu_1(1-g_1^N) \bigr)
$$
with $N=\ord q_{11}<\infty$ and $\mu_1\in \k$ as in Definition \ref{DefnLiftCoeffic}.
\end{prop}
These are the liftings of the Nichols algebra of Cartan type $A_1$ of Corollary \ref{CorNichAlgA1}.
As  concrete realizations for $\G=\Z/(N),\Z/(N^2)$ we name the following classic examples:

\begin{exs}Let $\operatorname{char} \k=0$.
\begin{enumerate}\item \emph{Taft Hopf algebra}. Let $\G:=\langle g_1\ |\  g_1^N=1\rangle\cong\Z/(N)$ and $\chi_1(g_1):=q\in\k^{\times}$ with $\ord q=N\ge 2$.
$$
 T(q)\ \cong\  \bigl(\k[x_1]\#\k[\G]\bigr)\big/(x_1^N).
$$ 
 \item \emph{Radford Hopf algebra}. Let $\G:=\langle g_1\ |\  g_1^{N^2}=1\rangle\cong\Z/(N^2)$ and $\chi_1(g_1):=q\in\k^{\times}$ with $\ord q=N\ge 2$. 
$$
 r(q)\cong (\k[x_1]\#\k[\G])/(x_1^N-(1-g_1^N)).
$$ 
\end{enumerate}
\end{exs}

\subsection{Pointed Hopf algebras of rank two}\label{SectExamplesRankTwo}
If $X=\{x_1,x_2\}$, then the situation is much more complicated such that we will treat here
 only to the well-known case $L=X$.

More complicated and new examples of pointed Hopf algebras for $L=\{x_1 , x_1x_2 , x_2\}$, 
$\{x_1 , x_1x_1x_2 , x_1x_2 , x_2\}$, etc.,
are found  in \cite{Helbig-PhD, Helbig-Lift} as liftings of Nichols algebras.

\begin{prop}\label{PropExamplesRank2}
Let $\operatorname{char} \k=p\ge 0$. Any character Hopf algebra $A$ of rank 2 with hard super letters $[L]=\{x_1,x_2\}$ is isomorphic to
\begin{align*}
 (\k\la x_1,x_2\ra\#\k[\G])/    (\ [x_1x_2]-\red{12}, \ \ x_1^{N_1}&-\redh{1},\\
                          x_2^{N_2}&-\redh{2}\ ),
\end{align*}
with $\red{12},\redh{1},\redh{2}$ as in Definition \ref{DefnIdealLcd}.
\end{prop}

\begin{proof} Since   $C(L)=\{x_1x_2\}$,  the claim follows by Theorem \ref{PropIdealCharHopfAlg}.
\end{proof}

As an example let us name the Nichols algebra of Cartan type $A_1\times A_1$ (quantum plane) and its liftings, with  the following concrete realizations:

\begin{exs} Let $\operatorname{char}\k=0$.
\begin{enumerate}
\item \emph{Nichols algebra of Cartan type $A_1\times A_1$}. Let $q_{12}q_{21}=1$, and $N_i=\ord q_{ii}\ge 2$, $i=1,2$.
\begin{align*}
\BV\cong \k\la x_1,x_2\ra/\bigl(\
     [x_1x_2], \ x_1^{N_1},\
                x_2^{N_2}\ \bigr).
\end{align*}

 \item \emph{Liftings of Cartan type $A_1\times A_1$}. Let $q_{12}q_{21}=1$, and $N_i=\ord q_{ii}\ge 2$, $i=1,2$.
\begin{align*}
 \bigl(\k\la x_1,x_2\ra\#\k[\G]\bigr)/\bigl(\quad
     [x_1x_2]&-\lambda_{12}(1-g_{1}g_{2}), & x_1^{N_1}&-\mu_1(1-g_1^{N_1}),\\
             &                             & x_2^{N_2}&-\mu_2(1-g_2^{N_2})\quad \bigr).
\end{align*}

\item \emph{Book Hopf algebra}. Let $q\in\k^{\times}$ with $\ord q=N\ge 2$, $\G:=\langle g \ |\ g^N=1\rangle\cong \Z/(N)$, $g_1:=g_2:=g$, and $\chi_1(g):=q^{-1}$, $\chi_2(g):=q$.
$$
h(1,q)\cong \bigl(\k\la x_1,x_2\ra\#\k[\G]\bigr)/\bigl(\ [x_1x_2],\ x_1^{N},\ x_2^{N}\ \bigr).
$$
\item \emph{Quantum groups}. $U_q(\mathfrak{sl}_2)$ and $u_q(\mathfrak{sl}_2)$ of Section \ref{SectionMotivation}.
\end{enumerate}
\end{exs}

\bibliographystyle{plain}
\bibliography{mybib}
\end{document}